\documentclass[12pt, reqno]{amsart}
\usepackage{amsmath, amsthm, amscd, amsfonts, amssymb, graphicx, color, mathrsfs}
\usepackage[bookmarksnumbered, colorlinks, plainpages]{hyperref}
\usepackage[all]{xy} 
\usepackage{slashed}

\newtheorem{theorem}{Theorem}[section]

\newtheorem{corollary}[theorem]{Corollary}

\theoremstyle{definition}

\theoremstyle{remark}
\newtheorem{remark}[theorem]{Remark}
\numberwithin{equation}{section}

\begin{document}
\setcounter{page}{1}

\title[$L^2$-Maximal functions on graded Lie groups]{$L^2$-Maximal functions on graded Lie groups}

\author[D. Cardona]{Duv\'an Cardona}
\address{
  Duv\'an Cardona:
  \endgraf
  Department of Mathematics: Analysis, Logic and Discrete Mathematics
  \endgraf
  Ghent University, Belgium
  \endgraf
  {\it E-mail address} {\rm duvanc306@gmail.com, duvan.cardonasanchez@ugent.be}
  }

\thanks{The author was supported by the Research Foundation-Flanders
(FWO) under the postdoctoral
grant No 1204824N.  The author was supported  by the FWO  Odysseus  1  grant  G.0H94.18N:  Analysis  and  Partial Differential Equations and by the Methusalem programme of the Ghent University Special Research Fund (BOF)
(Grant number 01M01021).
}


\begin{abstract} Bourgain  in his seminal paper \cite{Bourgain1986} about the analysis of maximal functions associated to convex bodies,  has estimated in a sharp way the $L^2$-operator norm of the maximal function associated to a kernel $K\in L^1,$ with differentiable Fourier transform $\widehat{K}.$  We formulate the extension to Bourgain's $L^2$-estimate in the setting of maximal functions on graded Lie groups. Our criterion is formulated in terms of the group Fourier transform of the kernel. We discuss the application of our main result to the $L^p$-boundedness of maximal functions on graded Lie groups. 
\end{abstract} 

\maketitle

\allowdisplaybreaks
\tableofcontents

\section{Introduction}
Let $G$ be a homogeneous (nilpotent) Lie group and let us consider a measurable function $K\in L^1(G).$ Extending a result due to Bourgain  \cite{Bourgain1986} we estimate the constant $\Gamma(K)$ that makes valid the $L^2$-inequality
\begin{equation}
    \Vert \sup_{t>0}|f\ast K_{t}|\Vert_{L^2(G)}\lesssim \Gamma(K)\Vert f \Vert_{L^2(G)},\,f\in C^{\infty}_0(G).
\end{equation}The quantity $\Gamma(K)$ involves information of the group Fourier transform of the kernel $K.$ Here, $K_t(x):=t^{-Q}K\circ D_{t^{-1}}(x)$ where $Q$ is the homogeneous dimension of the group $G,$ and $\{D_t\}_{t>0},$ denotes its  family of dilations. Since our criteria are presented in terms of the symbols of Rockland operators we consider $G$ to be a graded Lie group. Our main result is Theorem \ref{main:theorem} of Subsection \ref{Section:mainresult}. To illustrate our main result we are going to recall some aspects related to Bourgain's $L^2$-estimate \cite{Bourgain1986} as well as some of its consequences in the analysis of maximal functions on $\mathbb{R}^n$. 

\subsection{$L^2$-estimates for maximal functions on $\mathbb{R}^n$}

Bourgain in \cite{Bourgain1986} proved  that the maximal function 
\begin{equation}
    M_{B}f(x)=\sup_{r>0}\frac{1}{\textnormal{Vol}(r\cdot B)}\smallint_{r\cdot B}|f(x+y)|dy,
\end{equation}
associated to a centrally symmetric convex body in $\mathbb{R}^n,$ satisfies the $L^2$-operator norm estimate
\begin{equation}\label{B}
    \Vert M_Bf \Vert_{L^2(\mathbb{R}^n)}\leq D\Vert f\Vert_{L^2(\mathbb{R}^n)},
\end{equation}for a numerical constant $D>0$  independent of $B$ and of the dimension $n$. Bourgain's estimate in \eqref{B} extended in this general setting the celebrated results by Stein and Stein and Str\"omberg \cite{Stein1983:1,Stein1983:2} in the $L^2$-case. For a discussion on the subject we refer to the introduction of \cite[Pages 60-63]{BourgainMirekSteinWróbel}. Bourgain's argument for the proof of \eqref{B} combines two methods, the first one from Fourier analysis when estimating the $L^2$-operator norm of an arbitrary maximal function associated to an integrable kernel with differentiable Fourier transform and, the second one, from the geometrical properties of convex sets as Brunn's theorem, see e.g. Gromov and Milman \cite{Gromov:Milman}.  

The $L^2$-estimate for arbitrary maximal functions $M_Kf(x):=\sup_{t>0}|f\ast K_t(x)|$  due to Bourgain \cite{Bourgain1986} can be stated as follows. Here, we have used the notation $K_t(x)=t^{-n}K(t^{-1}x)$ for any function in the family $\{K_t\}_{t>0},$ among other things, preserving the $L^1$-norm of $K.$ Also, $\widehat{K}(\xi)=\smallint e^{-2\pi i x\cdot \xi}f(x)dx$ denotes the Fourier transform of $K.$
\begin{theorem}[Bourgain \cite{Bourgain1986}]\label{Bourgain:l2:estimate}
    Let $K\in L^1(\mathbb{R}^n)$ be such that its Fourier transform $\widehat{K}$ is differentiable. Define the following quantities
    \begin{equation}
        \alpha_j=\sup_{|\xi|\sim 2^{j}}|\widehat{K}(\xi)|,\quad \beta_j=\sup_{|\xi|\sim 2^{j}}|\langle\nabla \widehat{K}(\xi),\xi \rangle|.\footnote{Here, $\nabla K=(\partial_{x_i}K)_{1\leq i\leq n},$ denotes the gradient of $K,$ and the bracket $\langle x,y\rangle=\sum_{i=1}^{n}x_{i}y_{i},$ denotes the inner product on $\mathbb{R}^n.$}
    \end{equation}Then, we have the following estimate on the maximal operator associated to $K,$
    \begin{equation}\label{Sharp:estimate}
        \Vert \sup_{t>0}|f\ast K_t|  \Vert_{L^2(\mathbb{R}^n)}\leq C\Gamma(K)\Vert f \Vert_{L^2(G)},
    \end{equation}
    for every $f\in \mathscr{S}(\mathbb{R}^n),$ where 
    \begin{equation}\label{Gamma:K}
       \Gamma(K)=\sum_{j\in \mathbb{Z}}\alpha_j^{\frac{1}{2}}(\alpha_j+\beta_j)^{\frac{1}{2}}. 
    \end{equation}
\end{theorem} There are several reasons making Theorem \ref{Bourgain:l2:estimate} interesting by itself. First, the constant $\Gamma(K)$ in \eqref{Gamma:K} is sharp in some sense, see \cite[Page 1475]{Bourgain1986}. On the other hand, it has been useful in the analysis of the $L^2$-theory of maximal functions, for convex bodies and for the sphere $\mathbb{S}^{n-1}\subset \mathbb{R}^{n}$, as we describe in the following remarks. 
\begin{remark}
Let $\chi_A$ be the characteristic function of an arbitrary set $A\subset \mathbb{R}^n.$ A direct application of Theorem \ref{Bourgain:l2:estimate} to $K(x)=\chi_B-P_L$\footnote{Here, $\{P_t\}_{t>0}$ denotes the Poisson kernel defined via $\widehat{P}_t(\xi)=e^{-t|\xi|}$, $t=L(B)$ is a suitable constant introduced by Milman, see \cite[Page 1469]{Bourgain1986} for details.} provides the existence of the numerical constant $D$ in \eqref{B}, see \cite[Page 1473]{Bourgain1986}.    
\end{remark}
\begin{remark}
Consider the full maximal function on $\mathbb{R}^n$
\begin{equation}\label{Full}
   {M}^{d\sigma}_{F}f(x)=\sup_{r>0}\left|\smallint\limits_{\mathbb{S}^{n-1}}f(x-ry)d\sigma(y)\right|,\,n\geq 2,
\end{equation}where $d\sigma$ is the surface measure on the sphere $\mathbb{S}^{n-1}.$ For $n\geq 3,$ by an appropriate splitting procedure $d\sigma=d\sigma_1+d\sigma_2,$ and denoting by $M_1$ and $M_2$ the corresponding maximal functions for the measures   $d\sigma_1$ and $d\sigma_2,$ respectively, the $L^2$-theory of $M_2$ has been established in \cite{Bourgain1985}, leading to the proof of the restricted weak $(\frac{n}{n-1},1)$ inequality for ${M}^{d\sigma}_{F}.$ Note that the interpolation of the equivalent to this result in the setting of Lorentz spaces (namely, the boundedness of ${M}^{d\sigma}_{F}:L^{p,1}\rightarrow L^{p,\infty}$ with $p=\frac{n}{n-1}$) with the boundedness of ${M}^{d\sigma}_{F}$ on $L^{\infty}$ provides an alternative proof  of  the $L^p$-boundedness result due to Stein  \cite{Stein1976}, for the spherical maximal function  ${M}^{d\sigma}_{F}$ in the range $p>\frac{n}{n-1}.$    
\end{remark}
\begin{remark}
  Bourgain in  \cite{Bourgain1986:2} proved the $L^p$-boundedness of the spherical maximal function  ${M}^{d\sigma}_{F}$ in the lower dimension case  $n=2,$ for all $p>2.$ As it was pointed out in  \cite[Page 70]{Bourgain1986:2} the $L^2$-boundedeness of ${M}^{d\sigma}_{F}$ fails, but the analysis of its $L^2$-behaviour is important to understand its $L^p$-behaviour. Moreover, if one consider the regularised version
  \begin{equation}\label{Full:pt}
   {M}^{d\sigma}_{\varepsilon}f(x)=\sup_{r>0}\left|\smallint\limits_{\mathbb{S}^{n-1}}f(x-ry)(d\sigma\ast P_{\varepsilon})(y)\right|,\,n\geq 2,
\end{equation} the $L^2$-operator norm of ${M}^{d\sigma}_{\varepsilon}$ can be estimated as follows
\begin{equation}
    \Vert {M}^{d\sigma}_{\varepsilon} f\Vert_{L^2(\mathbb{R}^n)}\lesssim \log\left(\frac{1}{\varepsilon}\right)\Vert f\Vert_{L^2(\mathbb{R}^n)}.
\end{equation}Bourgain proved this inequality of norms using Theorem \ref{Bourgain:l2:estimate}, see \cite[Page 70]{Bourgain1986:2}.
\end{remark}
\begin{remark}
    Theorem \ref{Bourgain:l2:estimate} has been fundamental as well its consequences in further works about the $L^2$-theory of maximal functions in the Euclidean setting. For instance, see the references \cite{Bourgain2014,BourgainMirekSteinWróbel,Carbery1986}. In the next subsection we present our extension to Bourgain's estimate \eqref{Sharp:estimate} in Theorem \ref{Bourgain:l2:estimate} to the setting of graded Lie groups.
\end{remark}

\subsection{$L^2$-estimates for maximal functions on graded Lie groups}\label{Section:mainresult}
The analysis of the $L^p$-boundnedness of maximal functions on graded Lie groups and more general homogeneous Lie groups has been a problem of wide interest with a long tradition, see Folland and Stein \cite{FollandStein1982}. There has been several criteria for establishing these $L^p$-boundendness results. On graded Lie groups criteria in terms of the symbols of Rockland operators have been introduced in the work \cite{CardonaDelgadoRuzhanskyDyadic} by the author and M. Ruzhansky for the dyadic maximal function and in \cite{Cardona2024:Jan} for the full maximal function.   The understanding of the $L^2$-theory for these operators is fundamental in Littlewood-Paley's theory in the setting of graded Lie groups, and the analysis of this problem is the main goal of this manuscript. 

To present our main Theorem \ref{main:theorem} we introduce the required notation. In what follows, $G$ denotes a graded Lie group. This means that its Lie algebra admits a graduation $\mathfrak{g}=\mathfrak{g}_1\oplus \mathfrak{g}_2\oplus \cdots \oplus \mathfrak{g}_s,$ with $[\mathfrak{g}_i,\mathfrak{g}_j]\subset \mathfrak{g}_{i+j}.$ When $i+j>s,$ $\mathfrak{g}_{i+j}=\{0\},$ see Section \ref{Preliminaries} for details. Here, $\mathcal{R}$ denotes an arbitrary positive Rockland operator which, by definition, is a left-invariant  positive hypoelliptic partial differential operator on $G.$ We denote by $\nu>0$ the homogeneity degree of $\mathcal{R}.$  Here $\eta\in C^{\infty}_0(\mathbb{R})$ denotes a compactly supported function on the interval $I_\nu=[2^{-\nu},2^{\nu}].$ For every $j\in \mathbb{Z},$ define $\eta_{j}(t):=\eta(2^{-(j+1)}t).$ We assume that $\eta$ generates a partition of unity, namely, that
$ 
        \sum_{j\in \mathbb{Z}}\eta_{j}(t)=1,
$ $\forall t>0.$ 
The operator $\eta_j(\mathcal{R})$ is defined by the spectral calculus of the Rockland operator $\mathcal{R}$. Let $\widehat{G}$ be the unitary dual of $G.$ Each operator $\eta_j(\mathcal{R})$ is a Fourier multiplier on $L^2(G)$ and $\eta_j(\pi(\mathcal{R})),$ $\pi\in \widehat{G},$ denotes its symbol. Our main theorem is the following extension to the aforementioned result due to Bourgain.

\begin{theorem}\label{main:theorem}
    Consider a kernel $K\in L^1(G)$  on a graded Lie group $G.$ Define the quantities
    \begin{equation}\label{alpha:j:beta:j}
       \alpha_j= \sup_{\pi\in \widehat{G}}\Vert \widehat{K}(\pi)\eta_j(\pi(\mathcal{R}))\Vert_{\textnormal{op}},\quad \beta_j=\sup_{s>0;\pi\in \widehat{G}}\Vert s\frac{d}{ds}\{\widehat{K}(s\cdot \pi )\}\eta_j((s\cdot\pi)(\mathcal{R}))\Vert_{\textnormal{op}}.
    \end{equation}Then, we have the following estimate on the maximal operator associated to $K,$ 
    \begin{equation}
        \Vert \sup_{t>0}|f\ast K_t|  \Vert_{L^2(G)}\leq C\Gamma(K)\Vert f \Vert_{L^2(G)},
    \end{equation}
    for every $f\in \mathscr{S}(G),$ where 
    \begin{equation}\label{Gamma:K:2}
       \Gamma(K)=\sum_{j\in \mathbb{Z}}\alpha_j^{\frac{1}{2}}(\alpha_j+\beta_j)^{\frac{1}{2}}. 
    \end{equation}
\end{theorem}
We note that the smoothness condition on $\eta$ can be removed. Indeed, one can assume that $\eta$ is a continuous piecewise function generating a partition of unity, see Remark \ref{remark}.  Theorem \ref{main:theorem} applied to the Euclidean setting, that is $(\mathbb{R}^n,+)$ with the usual isotropic dilation, and with $\mathcal{R}=-\Delta_x$ being the positive Laplacian recovers Theorem \ref{Bourgain:l2:estimate}.   Indeed, note that in the Euclidean setting $\pi(\mathcal{R})$ is replaced by $|\xi|^2$ where $\xi\in \mathbb{R}^n,$ is the Fourier variable. We note that Theorem \ref{main:theorem} provides substantial information on the Littlewood-Paley components of the Fourier transform $\widehat{K}$ related to the Rockland operator $\mathcal{R}.$ Note also, that the hypothesis $K\in L^1(G),$ implies that the corresponding maximal operator $M_{K}:f\mapsto \sup_{t>0}|f\ast K_t|,$ is bounded on $L^{\infty}(G)$\footnote{Indeed, note that $|f\ast K_t(x)|\leq \Vert f\ast K_t\Vert_{L^\infty(G)}\leq \Vert f\Vert_{L^\infty(G)}\Vert K_t\Vert_{L^{1}(G)}= \Vert f\Vert_{L^\infty(G)}\Vert K\Vert_{L^{1}(G)},$ in view of Young's inequality. From where one deduces that $\Vert M_K f\Vert_{L^{\infty}(G)}\leq \Vert K\Vert_{L^1(G)}\Vert f\Vert_{L^{\infty}(G)}.$}. The Riesz-Thorin interpolation theorem implies the following $L^p$-boundedness result.
\begin{corollary} Consider a kernel $K\in L^1(G)$  on a graded Lie group $G$  and assume that $\Gamma(K)$ in \eqref{Gamma:K:2} is finite. Then, we have the following $L^p$-estimate on the maximal operator associated to $K,$
\begin{equation}
        \Vert \sup_{t>0}|f\ast K_t|  \Vert_{L^p(G)}\leq C_p\Gamma(K)^{\frac{2}{p}}\Vert K\Vert_{L^1(G)}^{1-\frac{2}{p}}\Vert f \Vert_{L^p(G)},
    \end{equation} for every $f\in \mathscr{S}(G),$ and for all $2\leq p\leq \infty.$
\end{corollary}
This paper is organised as follows. In Section \ref{Preliminaries} we present some basics about the Fourier analysis on nilpotent Lie groups and the spectral calculus of Rockland operators. In Section \ref{Proof:section} we present the the proof of our main Theorem \ref{main:theorem}.

\subsubsection*{Notation and remarks} We write $A\lesssim B$ when $A$ is less than $B$ up to a constant and $A\asymp B$ (or $A\sim B$) if $A\lesssim B$ and $B\lesssim A.$ For a bounded linear operator $T:H\rightarrow H$ defined on a Hilbert space $(H,\Vert\cdot\Vert_{H}),$ the quantity $\Vert T\Vert_{\textnormal{op}}=\sup_{\Vert x\Vert_{H}\leq 1}\Vert T(x)\Vert_{H}$ denotes its operator norm.  On a Lie group $G,$  we write $(x,y)\mapsto xy$ for the group operation and $x\mapsto x^{-1}$ for the inversion mapping. Here $e=e_{G}$ denotes the neutral element of $G.$   The Lie algebra of $G$ will be denoted by $\mathfrak{g}.$ The Haar measure of a nilpotent Lie group $G$  will be denoted by $dx.$ The corresponding $L^p$-spaces for the measure $dx$ will be denoted by $L^{p}(G)=L^p(G,\,dx).$ The case $p=2$ is crucial in this work. Since the group $G$ is nilpotent, in view of the identification $G\cong \mathbb{R}^n$ we can define the Schwartz class $\mathscr{S}(G)\cong \mathscr{S}(\mathbb{R}^n),$ and the corresponding space of distributions $\mathscr{D}'(G),$ see Section \ref{Preliminaries}.  A partial differential operator $P: C^\infty(G)\rightarrow \mathscr{D}'(G)$ is (globally) hypoelliptic, if the equation $Pu=f\in C^{\infty}(G)$ with $u\in \mathscr{D}'(G)$ imply that $u\in C^\infty(G),$ see H\"ormander \cite{HormanderBook34}.
The class of hypoelliptic differential operators of interest for this work is the family of Rockland operators, see Section \ref{Preliminaries}.

\section{Fourier analysis on nilpotent Lie groups and Rockland operators}\label{Preliminaries}

In this section, we present a consistent background of the Fourier analysis on graded Lie groups and on the spectral calculus of Rockland operators. We follow the notation related to the {\it quantisation on graded Lie groups} as in \cite{FischerRuzhanskyBook}. We introduce it as follows.

\subsubsection{Nilpotent Lie groups}\label{Nilpotent}
Let us consider a Lie group $G$ and let $(\mathfrak{g},[\cdot,\cdot])$ be its Lie algebra. Let us use the notation $\textnormal{ad}X(Y):=[X,Y]$ for the adjoint action $\textnormal{ad}: \mathfrak{g}\rightarrow \textnormal{End}(\mathfrak{g}).$ The Lie algebra $\mathfrak{g}$ is {\it nilpotent,} if for each $X\in \mathfrak{g},$ $(\textnormal{ad}X)^{k}=0$ for some $k\in \mathbb{N}.$ 
A {\it nilpotent Lie group} $G$ is a Lie group whose Lie algebra $\mathfrak{g}\cong T_{e}G$ is nilpotent.  Here we recall that the linear adjoint action on the Lie algebra $\mathfrak{g}$ of $G$ is denoted by $\textnormal{Ad}:G\rightarrow\textnormal{GL}(\mathfrak{g}) $ and defined as follows. Each element $g$ of  $G$ induces an inner automorphism of the Lie group $G$ by the formula $\textnormal{I}(g):x\in G\mapsto gxg^{-1}.$ Its differential $\textnormal{Ad}(g): \mathfrak{g}\rightarrow \mathfrak{g}$ gives an automorphism of the Lie algebra $\mathfrak{g}.$ The resulting representation $\textnormal{Ad}:G \rightarrow \textnormal{GL}(\mathfrak{g})$ is called  the adjoint representation of $G.$ The differential of the adjoint representation $\textnormal{Ad}:G\rightarrow \textnormal{GL}(\mathfrak{g})$ gives rise to to the  linear representation  $\textnormal{ad}: \mathfrak{g}\rightarrow \textnormal{End}(\mathfrak{g}),$ defined above as $\textnormal{ad}X(Y)=[X,Y].$ 

If $G$ is a connected and simply connected nilpotent Lie group then we have that
$\textnormal{exp}:\mathfrak{g}\rightarrow G$ is an analytic diffeomorphism. Among other things this diffeomorphism allows us to identify $G\cong \mathfrak{g}$ (and then $\mathfrak{g}\cong \mathbb{R}^{\dim(\mathfrak{g})}$) and define the Schwartz class $\mathscr{S}(G)\cong \mathscr{S}(\mathfrak{g}) $ and the space of distributions $\mathscr{D}'(G)$ using the exponential mapping.  
 
\subsubsection{The unitary dual $\widehat{G}$} Let $G$ be a simply connected nilpotent Lie group. We say that a mapping $\pi:G\rightarrow \textnormal{End}(H_{\pi}),$ for some separable Hilbert space $H_\pi,$ is a continuous, unitary and irreducible  representation of $G,$ if it satisfies the algebraic properties (i) and (iii), and the analytic property (ii) below:
\begin{itemize}
    \item[(i)] $\pi\in \textnormal{Hom}(G, \textnormal{U}(H_{\pi})),$  i.e. $\pi(xy)=\pi(x)\pi(y)$ and for the  adjoint of $\pi(x),$ $\pi(x)^*=\pi(x^{-1}),$ for every $x,y\in G.$ 
    \item[(ii)] The map $(x,v)\mapsto \pi(x)v, $ from $G\times H_\pi$ into $H_\pi$ is continuous. 
    \item[(iii)] For every $x\in G,$ and $W_\pi\subset H_\pi,$ if $\pi(x)W_{\pi}\subset W_{\pi},$ then $W_\pi=H_\pi$ or $W_\pi=\{0\}.$ 
\end{itemize}
Two unitary representations $ \pi\in \textnormal{Hom}(G,\textnormal{U}(H_\pi))$  and $\eta\in \textnormal{Hom}(G,\textnormal{U}(H_\eta))$  are {\it equivalent} if there exists a bounded linear mapping $T:H_\pi\rightarrow H_\eta$ such that for any $x\in G,$ $T\pi(x)=\eta(x)T.$ The mapping $T$ is called an intertwining operator between $\pi$ and $\eta.$ The relation $\sim$ on the set of unitary and irreducible representations $\textnormal{Rep}(G)$ defined by: {\it $\pi\sim \eta$ if and only if $\pi$ and $\eta$ are equivalent representations,} is an equivalence relation. The quotient set
$
    \widehat{G}:={\textnormal{Rep}(G)}/{\sim},
$ is called the {\it unitary dual} of $G.$

\subsubsection{Plancherel theorem on nilpotent Lie groups}  
The group Fourier transform of $f\in L^1(G), $  is defined by 
\begin{equation*}
   f\mapsto \mathscr{F}_{G}(f)\equiv \widehat{f},\,\quad  \widehat{f}(\pi):=\smallint\limits_{G}f(x)\pi(x)^*dx:H_\pi\rightarrow H_\pi,\quad \forall \pi \in\widehat{G}.
\end{equation*}
 The Schwartz space on the dual $\widehat{G}$ is defined by the image under the Fourier transform of the Schwartz space $\mathscr{S}(G),$ that is  $\mathscr{F}_{G}:\mathscr{S}(G)\rightarrow \mathscr{S}(\widehat{G}):=\mathscr{F}_{G}(\mathscr{S}(G)).$ 
If we identify a representation $\pi$ with its equivalence class, $[\pi]=\{\pi':\pi\sim \pi'\}$,  for every $\pi\in \widehat{G}, $ the Kirillov trace character $\Theta_\pi$ defined by (see Kirillov \cite{Kirillov1962})  $$[\Theta_{\pi},f]:
=\textnormal{Tr}(\widehat{f}(\pi)),$$ is a tempered distribution on $\mathscr{S}(G).$ The Fourier inversion formula can be written in terms of Kirillov's character as follows
\begin{equation*}
  \forall f\in L^1(G)\cap L^2(G),\,\, f(x)=\smallint\limits_{\widehat{G}}\textnormal{Tr}[\pi(x)\widehat{f}(\pi)]d\pi=:\mathscr{F}_{G}^{-1}[\widehat{f}\,](x),\,\,x\in G.
\end{equation*}
Here, $d\pi$ denotes the Plancherel measure on $\widehat{G}.$ The $L^2$-space on the dual is defined by the completion of $\mathscr{S}(G)$ with respect to the norm
\begin{equation}
    \Vert \sigma\Vert_{L^2(\widehat{G})}:=\left(\smallint\limits_{G}\|\sigma(\pi)\|_{\textnormal{HS}}^2d\pi \right)^{\frac{1}{2}},\,\,\sigma(\pi)\in \mathscr{S}(\widehat{G}),
\end{equation}where $\|\cdot\|_{\textnormal{HS}}$ denotes the Hilbert-Schmidt norm of operators on every representation space. The corresponding inner product on $L^2(\widehat{G})$ is given by
\begin{equation}
    (\sigma,\tau)_{L^2(\widehat{G})}:=\smallint\limits_{G}\textnormal{Tr}[\sigma(\pi)\tau(\pi)^{*}]d\pi,\,\,\sigma,\tau\in L^2(\widehat{G}),
\end{equation}where the notation $\tau(\pi)^{*}$ indicates the adjoint operator. Then,
the Plancherel theorem says that $\Vert f\Vert_{L^2(G)}=\Vert \widehat{f}\Vert_{L^2(\widehat{G})}$ for all $f\in L^2(G).$

An important subspace of each representation space $H_{\pi}$ is the one determined by the set of smooth vector $H_\pi^{\infty},$   that is, the space of vectors $v\in {H}_{\pi}$ such that the function $x\mapsto \pi(x)v,$ $x\in \widehat{G},$ is smooth. The space of smooth vectors $H_\pi^{\infty}$ is a dense sub-space of  ${H}_{\pi},$ see \cite[Page 42]{FischerRuzhanskyBook}. Finally, we record that the group convolution $f\ast g,$ of two functions, namely, the mapping defined via $f\ast g(x)=\smallint_Gf(xy^{-1})g(y)dy,$ $f,g\in L^{1}(G),$ under suitable conditions satisfies the identity $\widehat{g}(\pi)\widehat{f}(\pi)=\mathscr{F}_{G}[f\ast g](\pi).$

\subsubsection{Rockland operators on graded Lie groups}\label{Rockland:operators} A connected, simply connected nilpotent Lie group $G$ is {\it graded} if its Lie algebra $\mathfrak{g}$ may be decomposed as the direct sum of subspaces $  \mathfrak{g}=\mathfrak{g}_{1}\oplus\mathfrak{g}_{2}\oplus \cdots \oplus \mathfrak{g}_{s},$  such that the following bracket conditions are satisfied: 
$[\mathfrak{g}_{i},\mathfrak{g}_{j} ]\subset \mathfrak{g}_{i+j},$ where  $ \mathfrak{g}_{i+j}=\{0\}$ if $i+j\geq s+1.$ Each graded Lie group is nilpotent and there is a class of homogeneous operators that classifies the family of graded Lie groups among the family of nilpotent Lie groups, namely, the family of (left-invariant hypoelliptic partial differential operators, also called) {\it Rockland operators}.  We introduce this family in what follows.   Assume the Lie algebra $\mathfrak{g}$ to be endowed with a family of dilations $\{D_r\}_{r>0},$ compatible with its graduation, namely, for every $r>0,$ $D_{r}\in \textnormal{End}(\mathfrak{g}),$ and  $D_{r}$ is a mapping of the form
$ D_{r}=\textnormal{Exp}(\ln(r)A), $ 
for some diagonalisable linear operator $A\equiv \textnormal{diag}[\nu_1,\cdots,\nu_n]:\mathfrak{g}\rightarrow \mathfrak{g}.$ We call  the eigenvalues of the matrix $A,$ $\nu_1,\nu_2,\cdots,\nu_n,$ the {\it dilations weights} or {\it weights} of $G$. The homogeneous dimension of $G$ is given by 
$$  Q=\textnormal{Tr}(A)=\nu_1+\cdots+\nu_n.  $$ Note that the identification $G\cong \mathfrak{g}$ induced by the exponential coordinates allows us to define a family of dilations on $G$ that we denote also by $\{D_{r}\}_{r>0}.$ We write $D_r(x)=r\cdot x,$ $\forall x\in G,$ to simplify the notation. Then,
we record that a continuous linear operator $T:C^\infty(G)\rightarrow C^\infty(G)$ is homogeneous of  degree $\nu_T\in \mathbb{C}$ if, for every $r>0,$ the following equality holds 
\begin{equation*}
T(f\circ D_{r})=r^{\nu_T}(Tf)\circ D_{r},\,\forall f\in C^{\infty}(G),\,\forall r>0.
\end{equation*}
Let us consider $\pi\in\widehat{G}.$ We denote by ${H}_{\pi}^{\infty}$ the set of smooth vectors, also called G\r{a}rding vectors, formed by those $v\in {H}_{\pi},$ such that the function $x\mapsto \pi(x)v,$ $x\in \widehat{G},$ is smooth.  Then, a {\it Rockland operator} is a left-invariant partial  differential operator $$ \mathcal{R}=\sum_{[\alpha]=\nu}a_{\alpha}X^{\alpha}:C^\infty(G)\rightarrow C^{\infty}(G),\quad [\alpha]:=\sum_{i=1}^{n}\alpha_{i}\nu_i,$$  which is homogeneous of positive degree $\nu=\nu_{\mathcal{R}}$ and such that, it satisfies the {\it Rockland condition}, namely, for every unitary irreducible non-trivial representation $\pi\in \widehat{G},$ its symbol $\pi(\mathcal{R})$ defined via the Fourier inversion formula by
\begin{equation}\label{Symbol:R}
   \mathcal{R}f(x)= \smallint\limits_{\widehat{G}}\textnormal{Tr}[\pi(x)\pi(\mathcal{R})\widehat{f}(\pi)]d\pi,\,\,x\in G,
\end{equation}
is injective on ${H}_{\pi}^{\infty}.$ Note that $\sigma_{\mathcal{R}}(\pi)=\pi(\mathcal{R})$ coincides with the infinitesimal representation of $\mathcal{R}$ as an element of the universal enveloping algebra $\mathfrak{U}(\mathfrak{g})$.

If the Rockland operator $\mathcal{R}$ is symmetric, then $\mathcal{R}$ and $\pi(\mathcal{R})$ admit self-adjoint extensions on $L^{2}(G)$ and ${H}_{\pi},$ respectively.
Now if we preserve the same notation for their self-adjoint
extensions and we denote by $E_{\mathcal{R}}$ and $E_{\pi(\mathcal{R})}$  their spectral measures, we will denote by
\begin{equation}\label{Functional:identities:spectral:calculus}
    \psi(\mathcal{R})=\smallint\limits_{-\infty}^{\infty}\psi(\lambda) dE_{\mathcal{R}}(\lambda),\,\,\,\textnormal{and}\,\,\,\pi(\psi(\mathcal{R}))\equiv \psi(\pi(\mathcal{R}))=\smallint\limits_{-\infty}^{\infty}\psi(\lambda) dE_{\pi(\mathcal{R})}(\lambda),
\end{equation}
the functions defined by the functional calculus. In general, we will reserve the notation $\{dE_A(\lambda)\}_{0\leq\lambda<\infty}$ for the spectral measure associated with a positive and self-adjoint operator $A$ on a Hilbert space $H.$ Moreover if $f\in L^{\infty}(\mathbb{R}^{+}_0),$ we have that
\begin{equation}\label{op:inequality}
    \Vert f(A)\Vert_{\textnormal{op}}=\left\Vert \smallint_{0}^{\infty}f(\lambda)dE_A(\lambda) \right\Vert_{\textnormal{op}}\leq \Vert f\Vert_{L^{\infty}(\mathbb{R}^{+}_0)}.
\end{equation}

\section{Estimation of maximal functions using the Fourier transform}\label{Proof:section}

In this section we present the proof of Theorem \ref{main:theorem}.
Before presenting our main result we introduce an additional notation and recall some fundamental facts. 
For every $\pi\in \widehat{G},$ let us define for any $r>0,$ and any $x\in G,$ 
$(r\cdot \pi)(x):=\pi(r\cdot x). $ 
 Then, for each test function $f\in C^{\infty}_0(\mathbb{R}^+_0),$  $f( r\cdot \pi(\mathcal{R}))=f({r^{\nu}\pi(\mathcal{R})}),$ see Lemma 4.3 of \cite{FischerRuzhanskyBook}. We  use the notation 
\begin{equation}\label{start:notation}
    K_{t}:=t^{-Q}K(t^{-1}\cdot),\quad t>0,
\end{equation} and the following identity for the Fourier transforms $\widehat{K}_t$ and $\widehat{K},$ of $K_t$ and $K,$ respectively,
\begin{equation}\label{Eq:dilatedFourier}
    \widehat{K}_{t}(\pi)=\smallint\limits_{G}t^{-Q}K(t^{-1}\cdot x)\pi(x)^*dx=\smallint\limits_{G}K(y)\pi(t\cdot y)^{*}dy=\widehat{K}(t\cdot \pi),
\end{equation}always that $K\in L^1(G).$ In \eqref{Eq:dilatedFourier} we have applied the changes of variables $t^{-1}\cdot x\mapsto y,$ and we have used the new volume element $dy=t^{-Q}dx.$ We also note that the mapping $\frac{d}{ds}\{\widehat{K}(s\cdot \pi )\},$ can be defined as a densely defined operator on each representation space $H_{\pi},$ taking the set of smooth vectors $H_\pi^{\infty}$ as its domain, see e.g. \cite[Section 3]{Cardona2024:Jan}.
\begin{remark}\label{remark}
    Note  that $\alpha_j=\alpha_{j}^{(\eta)}$ and $\beta_{j}^{(\eta)}$ defined by 
 \begin{equation}
       \alpha_j= \sup_{\pi\in \widehat{G}}\Vert \widehat{K}(\pi)\eta_j(\pi(\mathcal{R}))\Vert_{\textnormal{op}},\quad \beta_j=\sup_{s>0;\pi\in \widehat{G}}\Vert s\frac{d}{ds}\{\widehat{K}(s\cdot \pi )\}\eta_j((s\cdot\pi)(\mathcal{R}))\Vert_{\textnormal{op}},
    \end{equation}initially are quantities depending on $\eta.$ However, if we consider $\phi$ being a positive and continuous piecewise function on $\mathbb{R},$ supported in a closed interval $I=[c,d],$ and generating a partition of unity
     $ 
        \sum_{j\in \mathbb{Z}}\phi_{j}(t)=1,\,\phi_{j}(t)=\phi(2^{-(j+1)}t), \,t>0,
    $  and we define
\begin{equation}
       \alpha_j^{(\phi)}= \sup_{\pi\in \widehat{G}}\Vert \widehat{K}(\pi)\phi_j(\pi(\mathcal{R}))\Vert_{\textnormal{op}},\quad \beta_j^{(\phi)}=\sup_{s>0;\pi\in \widehat{G}}\Vert s\frac{d}{ds}\{\widehat{K}(s\cdot \pi )\}\phi_j((s\cdot\pi)(\mathcal{R}))\Vert_{\textnormal{op}},
    \end{equation}we have that $\alpha_j^{(\phi)}\asymp \alpha_{j}^{(\eta)},$  $\beta_j^{(\phi)}\asymp \beta_{j}^{(\eta)}.$  So, in view of these estimates for the constants $\alpha_{j}^{(\eta)},\beta_{j}^{(\eta)},$ we simplify the notation and we just write  $\alpha_{j},\beta_{j},$ as in \eqref{alpha:j:beta:j}, for any $j\in \mathbb{Z}$. We continue this discussion, namely, the proof of the previous fact in Remark \ref{remark:2}.
\end{remark}

\begin{proof}[Proof of Theorem \ref{main:theorem}]
We start the proof by considering the partition of unity for the spectrum of the Rockland operator $\mathcal{R}$ generated by $\eta\in C^\infty_0(\mathbb{R}),$  
\begin{equation}\label{ref:R:D}
 \forall\lambda\in (0,\infty),\,  \sum_{j=-\infty}^\infty\eta(2^{j\nu}\lambda)=1.
\end{equation}We have that $\textnormal{supp}(\eta)\subset [1/2^{\nu},2^{\nu}].$ As a consequence of \eqref{ref:R:D}, the spectral calculus of Rockland operators implies that
$ 
    \sum_{j=-\infty}^\infty\eta(2^{j\nu}\mathcal{R})=I,
$ as well as the convergence of $ 
    \sum_{j=-\infty}^\infty\eta(2^{j\nu}\mathcal{R})\delta=\delta,
$ in $\mathscr{S}'(G)$ to the Dirac distribution $\delta.$
By defining 
\begin{equation}
    \eta_j(\mathcal{R})=\smallint_{0}^\infty\eta_{j}(\lambda)dE_{\mathcal{R}}(\lambda),\,\,\eta_j(\lambda)=\eta(2^{-(j+1)\nu}\lambda),\, \textnormal{supp}(\eta_j)\subset [2^{j\nu},2^{(j+2)\nu}],
\end{equation} note that
\begin{equation}
    \widehat{\eta_j(\mathcal{R})\delta }(\pi)=\eta((2^{-(j+1)}\cdot \pi)(\mathcal{R})),
\end{equation} and with $k_j$ defined by
\begin{equation}
   \widehat{k}_{j}(\pi)=\widehat{K}(\pi)\eta_j(\pi(\mathcal{R}))=\widehat{K}(\pi)\eta(2^{-(j+1)}\cdot\pi)(\mathcal{R}),
\end{equation}
we have that 
\begin{align*}
    \Vert \sup_{t>0}|f\ast k_t|\Vert_{L^2(G)}\leq \sum_{j\in \mathbb{Z}}\Vert \sup_{t>0}|f\ast (k_j)_t| \Vert_{L^2(G)}.
\end{align*}Let us fix $j\in \mathbb{Z}.$ We have that 
\begin{equation}\label{First:Ineq}
    \sup_{t>0}|f\ast (k_j)_t(x)| \leq \left(\sum_{\ell\in \mathbb{Z}} \sup_{2^{\ell}\leq t<2^{\ell+1}}|f\ast (k_j)_t(x)|^2\right)^{\frac{1}{2}}.
\end{equation}
In our further analysis we are going to estimate the $L^2$-norm of the right hand side of \eqref{First:Ineq} and later we make the summation  of these norms over $j\in \mathbb{Z}.$
In view of the Fourier inversion formula can write 
\begin{align*}
  f\ast (k_j)_t(x)=\smallint_{\widehat{G}}\textnormal{Tr}[\pi(x)\widehat{k}_{j}(t\cdot \pi)\widehat{f}(\pi)]=\smallint_{\widehat{G}}\textnormal{Tr}[\pi(x)\widehat{K}(\pi)\eta([2^{-(j+1)}\cdot t\cdot \pi](\mathcal{R}))\widehat{f}(\pi)].  
\end{align*}
Fix $t>0$ such that $2^{\ell}\leq t<2^{\ell+1}.$ The spectral calculus for Rockland operators allows us to write 
\begin{align*}
    \eta([2^{-(j+1)}\cdot t\cdot \pi](\mathcal{R}))\widehat{f}(\pi)&=\smallint\eta(2^{-(j+1)\nu}t^{\nu}\lambda)dE_{\pi(\mathcal{R})}(\lambda)\\
    &=\smallint_{2^{-\nu}\leq 2^{-(j+1)\nu}t^{\nu}\lambda \leq 2^{\nu}}\eta(2^{-(j+1)\nu}t^{\nu}\lambda)dE_{\pi(\mathcal{R})}(\lambda)\\
    &=\smallint_{2^{(j-\ell-2)\nu}\leq \lambda\leq 2^{(j-\ell+2)\nu}}\eta(2^{-(j+1)\nu}t^{\nu}\lambda)dE_{\pi(\mathcal{R})}(\lambda).
\end{align*}
For $m\in \mathbb{Z},$ let us denote
\begin{equation}
    E_m(\pi):=E_{\pi(\mathcal{R})}[2^{(m-2)\nu},2^{(m+2)\nu}].
\end{equation}Observe that with $m=j-\ell,$
\begin{align*}
   & E_m(\pi)\eta([2^{-(j+1)}\cdot t\cdot \pi](\mathcal{R}))\widehat{f}(\pi)\\
   &= E_{\pi(\mathcal{R})}[2^{(m-2)\nu},2^{(m+2)\nu}]\smallint_{2^{(j-\ell-2)\nu}\leq \lambda\leq 2^{(j-\ell+2)\nu}}\eta(2^{-(j+1)\nu}t^{\nu}\lambda)dE_{\pi(\mathcal{R})}(\lambda)\widehat{f}(\pi)\\
    &=\smallint_{2^{(j-\ell-2)\nu}\leq \lambda\leq 2^{(j-\ell+2)\nu}}\eta(2^{-(j+1)\nu}t^{\nu}\lambda)dE_{\pi(\mathcal{R})}(\lambda)E_{\pi(\mathcal{R})}[2^{(m-2)\nu},2^{(m+2)\nu}] \widehat{f}(\pi)\\
    &=\eta([2^{-(j+1)}\cdot t\cdot \pi](\mathcal{R}))E_m(\pi)\widehat{f}(\pi).
\end{align*}In consequence, 
\begin{align*}
  f\ast (k_j)_t(x)=\smallint_{\widehat{G}}\textnormal{Tr}[\pi(x)\widehat{k}_{j}(t\cdot \pi)E_{j-\ell}(\pi)\widehat{f}(\pi)]d\pi=\mathscr{F}_{G}^{-1}[\widehat{k}_{j}(t\cdot \pi)E_{j-\ell}(\pi)\widehat{f}(\pi)].  
\end{align*} Let us fix an integer $A_j\geq 1,$ and for a fixed $\ell\in \mathbb{Z},$ let $\{t_{\tau}\}_{\tau\leq A_j}$ be a $(2^{\ell}A_j^{-1})$-net in the interval $I_\ell=[2^\ell, 2^{\ell+1}].$ Le us assume that this net also includes the point $t\in [2^\ell, 2^{\ell+1}].$ By a $(2^{\ell}A_j^{-1})$-net of $I_\ell,$ we mean a partition $P_\ell=\{t_{\tau}\}_{\tau\leq A_j}$ of $I_\ell$ such that the following properties are satisfied: 
\begin{itemize}
    \item the number of elements of $P_\ell,$ namely, its cardinality is $A_j.$
    \item The partition $P_\ell=\{t_{\tau}\}_{\tau\leq A_j}$ is ordered in such a way that two consecutive elements $t_{\tau}, t_{\tau+1},$ satisfy the distance condition $|t_{\tau}-t_{\tau+1}|\leq 2^{\ell}A_j^{-1}.$
\end{itemize}
Additionally, 
\begin{itemize}
    \item  without loss of generality we have assume that $t=t_{\tau_0}$ also belongs to the partition $P_\ell.$
\end{itemize}
The fundamental theorem of calculus allows us to write 
\begin{align*}
    f\ast (k_j)_t(x) &= \mathscr{F}_{G}^{-1}[\widehat{k}_{j}(t\cdot \pi)E_{j-\ell}(\pi)\widehat{f}(\pi)]\\
    &=\mathscr{F}_{G}^{-1}[\widehat{k}_{j}(t_{\tau_0-1}\cdot \pi)E_{j-\ell}(\pi)\widehat{f}(\pi)]+\smallint_{t_{\tau_0-1}}^{t_{\tau_0}}\frac{d}{ds}\mathscr{F}_{G}^{-1}[\widehat{k}_{j}(s\cdot \pi)E_{j-\ell}(\pi)\widehat{f}(\pi)]ds.
\end{align*}In view of triangle inequality, we have that 
\begin{align*}
  & |f\ast (k_j)_t(x) |\\
  &\leq |\mathscr{F}_{G}^{-1}[\widehat{k}_{j}(t_{\tau_0-1}\cdot \pi)E_{j-\ell}(\pi)\widehat{f}(\pi)]|+ \smallint_{t_{\tau_0-1}}^{t_{\tau_0}}|\frac{d}{ds}\mathscr{F}_{G}^{-1}[\widehat{k}_{j}(s\cdot \pi)E_{j-\ell}(\pi)\widehat{f}(\pi)]|ds\\
  &\leq |\mathscr{F}_{G}^{-1}[\widehat{k}_{j}(t_{\tau_0-1}\cdot \pi)E_{j-\ell}(\pi)\widehat{f}(\pi)]|+ \smallint_{t_{\tau_0-1}}^{t_{\tau_0}+1}|\frac{d}{ds}\mathscr{F}_{G}^{-1}[\widehat{k}_{j}(s\cdot \pi)E_{j-\ell}(\pi)\widehat{f}(\pi)]|ds.
\end{align*} In consequence, summing over the points of the partition $\tilde{P_\ell}=P_\ell\setminus\{t\},$ (and enumerating these points again in order to keep the notation $t_\tau,$ $\tau\leq A_j,$ for the points in $\tilde{P_\ell}$) we have that
\begin{align*}
    &|f\ast (k_j)_t(x) |\\
    &\leq \left( \sum_{\tau}|\mathscr{F}_{G}^{-1}[\widehat{k}_{j}(t_{\tau}\cdot \pi)E_{j-\ell}(\pi)\widehat{f}(\pi)]|^2  +\left[\smallint_{t_{\tau}}^{t_{\tau}+1}|\frac{d}{ds}\mathscr{F}_{G}^{-1}[\widehat{k}_{j}(s\cdot \pi)E_{j-\ell}(\pi)\widehat{f}(\pi)]|ds\right]^2\right)^{\frac{1}{2}}.
\end{align*} Let us observe that the right hand side of the previous inequality does not depend on $t.$   In consequence
\begin{align*}
   & \sup_{2^{\ell}\leq t<2^{\ell+1}}|f\ast (k_j)_t(x) |\\
    &\leq\left( \sum_{\tau} |\mathscr{F}_{G}^{-1}[\widehat{k}_{j}(t_{\tau}\cdot \pi)E_{j-\ell}(\pi)\widehat{f}(\pi)](x)|^2  +\left[\smallint_{t_{\tau}}^{t_{\tau}+1}|\frac{d}{ds}\mathscr{F}_{G}^{-1}[\widehat{k}_{j}(s\cdot \pi)E_{j-\ell}(\pi)\widehat{f}(\pi)](x)|ds\right]^2\right)^{\frac{1}{2}}.
\end{align*}Now, by applying both,  Plancherel theorem and Minkowski integral inequality observe that
\begin{align*}
  & \left\Vert \sup_{2^{\ell}\leq t<2^{\ell+1}}|f\ast (k_j)_t(x) |   \right\Vert_{L^2(G)}^2\\
  &\leq\smallint_{G} \sum_{\tau} |\mathscr{F}_{G}^{-1}[\widehat{k}_{j}(t_{\tau}\cdot \pi)E_{j-\ell}(\pi)\widehat{f}(\pi)](x)|^2 \\
  &\hspace{2cm} +\left[\smallint_{t_{\tau}}^{t_{\tau}+1}|\frac{d}{ds}\mathscr{F}_{G}^{-1}[\widehat{k}_{j}(s\cdot \pi)E_{j-\ell}(\pi)\widehat{f}(\pi)](x)|ds\right]^2 dx\\
 &=\sum_{\tau}\smallint_{G}  |\mathscr{F}_{G}^{-1}[\widehat{k}_{j}(t_{\tau}\cdot \pi)E_{j-\ell}(\pi)\widehat{f}(\pi)](x)|^2dx\\
 &\hspace{2cm}+\smallint_{G}\left[\smallint_{t_{\tau}}^{t_{\tau}+1}|\frac{d}{ds}\mathscr{F}_{G}^{-1}[\widehat{k}_{j}(s\cdot \pi)E_{j-\ell}(\pi)\widehat{f}(\pi)](x)|ds\right]^2dx  \\
 &\leq \sum_{\tau}\smallint_{G}  |\mathscr{F}_{G}^{-1}[\widehat{k}_{j}(t_{\tau}\cdot \pi)E_{j-\ell}(\pi)\widehat{f}(\pi)](x)|^2dx \\
 &\hspace{2cm}+\left[\smallint_{t_{\tau}}^{t_{\tau}+1}\left(\smallint_{G}|\frac{d}{ds}\mathscr{F}_{G}^{-1}[\widehat{k}_{j}(s\cdot \pi)E_{j-\ell}(\pi)\widehat{f}(\pi)](x)|^2dx\right)^{\frac{1}{2}} ds\right]^2\\
 &=\sum_{\tau}\smallint_{\widehat{G}}  \Vert[\widehat{k}_{j}(t_{\tau}\cdot \pi)E_{j-\ell}(\pi)\widehat{f}(\pi)\Vert_{\textnormal{HS}}^2dx \\
 &\hspace{2cm}+\left[\smallint_{t_{\tau}}^{t_{\tau}+1}\left(\smallint_{\widehat{G}}\Vert\frac{d}{ds}[\widehat{k}_{j}(s\cdot \pi)E_{j-\ell}(\pi)\widehat{f}(\pi)]\Vert_{\textnormal{HS}}^2d\pi\right)^{\frac{1}{2}} ds\right]^2\\
 &\leq \sum_{\tau}\smallint_{\widehat{G}}  \Vert[\widehat{k}_{j}(t_{\tau}\cdot \pi)E_{j-\ell}(\pi)\widehat{f}(\pi)\Vert_{\textnormal{HS}}^2d\pi \\
 &\hspace{2cm}+\left[\smallint_{t_{\tau}}^{t_{\tau}+1}\sup_{s\sim 2^{\ell}}\left(\smallint_{\widehat{G}}\Vert\frac{d}{ds}[\widehat{k}_{j}(s\cdot \pi)E_{j-\ell}(\pi)\widehat{f}(\pi)]\Vert_{\textnormal{HS}}^2d\pi\right)^{\frac{1}{2}} ds'\right]^2\\
 &\leq \sum_{\tau}\smallint_{\widehat{G}}  \Vert[\widehat{k}_{j}(t_{\tau}\cdot \pi)E_{j-\ell}(\pi)\widehat{f}(\pi)\Vert_{\textnormal{HS}}^2d\pi \\
 &\hspace{2cm}+\left[2^\ell A_{j}^{-1}\sup_{s\sim 2^{\ell}}\left(\smallint_{\widehat{G}}\Vert\frac{d}{ds}[\widehat{k}_{j}(s\cdot \pi)E_{j-\ell}(\pi)\widehat{f}(\pi)]\Vert_{\textnormal{HS}}^2d\pi\right)^{\frac{1}{2}} \right]^2\\
 &=A+B.
\end{align*} We shall estimate the first term $A.$ In order to simplify the notation let us define $f^{(m)}=\mathscr{F}_{G}^{-1}[E_{m}(\pi)\widehat{f}(\pi)],$ for any $m\in \mathbb{Z}.$ Note that 
\begin{align*}
   A&= \smallint_{\widehat{G}}  \Vert[\widehat{k}_{j}(t_{\tau}\cdot \pi)E_{j-\ell}(\pi)\widehat{f}(\pi)\Vert_{\textnormal{HS}}^2d\pi\\ &\leq \sup_{\pi'\in \widehat{G}}\Vert \widehat{k}_j(\pi')  \Vert_{\textnormal{op}}^2\Vert \widehat{f}^{(j-\ell)}(\pi)\Vert_{L^2(\widehat{G})}^2\\
     &=\sup_{\pi'\in \widehat{G}}\Vert K(\pi')\eta_j(\pi'(\mathcal{R}))\Vert_{\textnormal{op}}^2\Vert \widehat{f}^{(j-\ell)}(\pi)\Vert_{L^2(\widehat{G})}^2\\
     &=\alpha_j^2\Vert \widehat{f}^{(j-\ell)}(\pi)\Vert_{L^2(\widehat{G})}^2.
\end{align*}
In consequence
\begin{align*}
  A=  \smallint_{\widehat{G}}  \Vert[\widehat{k}_{j}(t_{\tau}\cdot \pi)E_{j-\ell}(\pi)\widehat{f}(\pi)\Vert_{\textnormal{HS}}^2d\pi \leq \alpha_j^2 \Vert \widehat{f}^{(j-\ell)}(\pi)\Vert_{L^2(\widehat{G})}^2.
\end{align*}On the other hand, let us estimate the term
\begin{equation}
 \sqrt{B}:=   2^\ell A_{j}^{-1}\sup_{s\sim 2^{\ell}}\left(\smallint_{\widehat{G}}\Vert\frac{d}{ds}[\widehat{k}_{j}(s\cdot \pi)E_{j-\ell}(\pi)\widehat{f}(\pi)]\Vert_{\textnormal{HS}}^2d\pi\right)^{\frac{1}{2}}.
\end{equation}
Since 
$$ \frac{d}{ds}[\widehat{k}_{j}(s\cdot \pi)]=\frac{d}{ds}\{\widehat{K}(s\cdot \pi )\}\eta_j((s\cdot\pi)(\mathcal{R}))+\widehat{K}(s\cdot \pi )\frac{d}{ds}\{\eta_j((s\cdot\pi)(\mathcal{R}))\}$$
we have that
\begin{align*}
    &\left(\smallint_{\widehat{G}}\Vert\frac{d}{ds}[\widehat{k}_{j}(s\cdot \pi)]\widehat{f}^{(j-\ell)}(\pi)\Vert_{\textnormal{HS}}^2d\pi\right)^{\frac{1}{2}}\leq \left(\smallint_{\widehat{G}}\Vert\frac{d}{ds}\widehat{K}(s\cdot \pi )\eta_j((s\cdot\pi)(\mathcal{R}))\widehat{f}^{(j-\ell)}(\pi)\Vert_{\textnormal{HS}}^2d\pi\right)^{\frac{1}{2}}\\
    &\hspace{2cm}+\left(\smallint_{\widehat{G}}\Vert\widehat{K}(s\cdot \pi )\frac{d}{ds}\eta_j((s\cdot\pi)(\mathcal{R}))\widehat{f}^{(j-\ell)}(\pi)\Vert_{\textnormal{HS}}^2d\pi\right)^{\frac{1}{2}}.
\end{align*}Let $\kappa(t)$ be defined by $\kappa(t)=\eta(t^\nu).$ Note that 
\begin{align*}
   \eta_j((s\cdot\pi)(\mathcal{R}))  &=\smallint_0^{\infty}\eta(2^{-j\nu}s^{\nu}\lambda)dE_{\pi(\mathcal{R})}(\lambda)=\smallint_0^{\infty}\kappa(2^{-j}s\lambda^{\frac{1}{\nu}})dE_{\pi(\mathcal{R})}(\lambda)\\
    &=\smallint_0^{\infty}\kappa(2^{-j}st)dE_{\pi(\mathcal{R})^{\frac{1}{\nu}}}(t).
\end{align*}
The function $\kappa$ allows us to avoid unbounded terms when computing the derivative $\frac{d}{ds}\eta_j((s\cdot\pi)(\mathcal{R})).$ To estimate it in an efficient way let us make use of the functional calculus for Rockland operators. Indeed,  
\begin{align*}
    \frac{d}{ds}\eta_j((s\cdot\pi)(\mathcal{R}))=\frac{d}{ds}\smallint_0^{\infty}\kappa(2^{-j}st)dE_{\pi(\mathcal{R})^{\frac{1}{\nu}}}(t)=\smallint_0^{\infty}\kappa'(2^{-j}st)2^{-j}tdE_{\pi(\mathcal{R})^{\frac{1}{\nu}}}(t)
\end{align*}
Now, we have that
\begin{align*}
   \sqrt{B}&= 2^\ell A_{j}^{-1}\sup_{s\sim 2^{\ell}}\left(\smallint_{\widehat{G}}\Vert\frac{d}{ds}[\widehat{k}_{j}(s\cdot \pi)E_{j-\ell}(\pi)\widehat{f}(\pi)]\Vert_{\textnormal{HS}}^2d\pi\right)^{\frac{1}{2}}\\
    &\lesssim 2^\ell A_{j}^{-1}\sup_{s\sim 2^{\ell}}\left(\smallint_{\widehat{G}}\Vert\frac{d}{ds}\widehat{K}(s\cdot \pi )\eta_j((s\cdot\pi)(\mathcal{R}))\widehat{f}^{(j-\ell)}(\pi)\Vert_{\textnormal{HS}}^2d\pi\right)^{\frac{1}{2}}\\
    &+2^\ell A_{j}^{-1}\sup_{s\sim 2^{\ell}}\left(\smallint_{\widehat{G}}\Vert\widehat{K}(s\cdot \pi )\frac{d}{ds}\eta_j((s\cdot\pi)(\mathcal{R}))\widehat{f}^{(j-\ell)}(\pi)\Vert_{\textnormal{HS}}^2d\pi\right)^{\frac{1}{2}}\\
    &\lesssim  2^\ell A_{j}^{-1}\sup_{s\sim 2^{\ell}}\left(\smallint_{\widehat{G}}\Vert\frac{d}{ds}\widehat{K}(s\cdot \pi )\eta_j((s\cdot\pi)(\mathcal{R}))\Vert_{\textnormal{op}}^2\Vert\widehat{f}^{(j-\ell)}(\pi)\Vert_{\textnormal{HS}}^2d\pi\right)^{\frac{1}{2}}\\
    &+2^\ell A_{j}^{-1}\sup_{s\sim 2^{\ell}}\left(\smallint_{\widehat{G}}\Vert\widehat{K}(s\cdot \pi )\frac{d}{ds}\eta_j((s\cdot\pi)(\mathcal{R}))\widehat{f}^{(j-\ell)}(\pi)\Vert_{\textnormal{HS}}^2d\pi\right)^{\frac{1}{2}}\\
    &=I+II.
\end{align*}In order to estimate $I$, observe that
\begin{align*}
    I&=  2^\ell A_{j}^{-1}\sup_{s\sim 2^{\ell}}\left(\smallint_{\widehat{G}}\Vert\frac{d}{ds}\widehat{K}(s\cdot \pi )\eta_j((s\cdot\pi)(\mathcal{R}))\Vert_{\textnormal{op}}^2\Vert\widehat{f}^{(j-\ell)}(\pi)\Vert_{\textnormal{HS}}^2d\pi\right)^{\frac{1}{2}}\\
     &\asymp 2^\ell A_{j}^{-1}\sup_{s\sim 2^{\ell}}s^{-1}\left(\smallint_{\widehat{G}}\Vert s\frac{d}{ds}\widehat{K}(s\cdot \pi )\eta_j((s\cdot\pi)(\mathcal{R}))\Vert_{\textnormal{op}}^2\Vert\widehat{f}^{(j-\ell)}(\pi)\Vert_{\textnormal{HS}}^2d\pi\right)^{\frac{1}{2}}\\
     &\lesssim 2^\ell A_{j}^{-1}\sup_{s\sim 2^{\ell}}s^{-1}\left(\smallint_{\widehat{G}}\sup_{s'>0;\pi\in \widehat{G}}\Vert s'\frac{d}{ds'}\widehat{K}(s'\cdot \pi )\eta_j((s'\cdot\pi)(\mathcal{R}))\Vert_{\textnormal{op}}^2\Vert\widehat{f}^{(j-\ell)}(\pi)\Vert_{\textnormal{HS}}^2d\pi\right)^{\frac{1}{2}}\\
     &\asymp  2^\ell A_{j}^{-1}2^{-\ell}\left(\smallint_{\widehat{G}}\beta_j^2\Vert\widehat{f}^{(j-\ell)}(\pi)\Vert_{\textnormal{HS}}^2d\pi\right)^{\frac{1}{2}}\\
     &= A_{j}^{-1}\beta_j\Vert \widehat{f}^{(j-\ell)}(\pi)\Vert_{L^2(\widehat{G})}.
\end{align*}For the analysis of the remainder term $II,$ consider the test function $\rho(\lambda):=\kappa'(\lambda^{\frac{1}{\nu}})\lambda^{\frac{1}{\nu}},$ $\lambda>0.$ Note that  $\rho(t^{\nu}):=\kappa'(t)t,$ for all $t>0.$ Now, we proceed as follows
\begin{align*}
    &2^\ell A_{j}^{-1}\sup_{s\sim 2^{\ell}}\left(\smallint_{\widehat{G}}\Vert\widehat{K}(s\cdot \pi )\frac{d}{ds}\eta_j((s\cdot\pi)(\mathcal{R}))\widehat{f}^{(j-\ell)}(\pi)\Vert_{\textnormal{HS}}^2d\pi\right)^{\frac{1}{2}}\\
    &=2^\ell A_{j}^{-1}\sup_{s\sim 2^{\ell}}\left(\smallint_{\widehat{G}}\Vert\widehat{K}(s\cdot \pi )\smallint_0^{\infty}\kappa'(2^{-j}st)2^{-j}tdE_{\pi(\mathcal{R})^{\frac{1}{\nu}}}(t)\widehat{f}^{(j-\ell)}(\pi)\Vert_{\textnormal{HS}}^2d\pi\right)^{\frac{1}{2}}\\
     &\asymp A_{j}^{-1}\sup_{s\sim 2^{\ell}}\left(\smallint_{\widehat{G}}\Vert\widehat{K}(s\cdot \pi )\smallint_0^{\infty}\kappa'(2^{-j}st)2^{-j}\times s\times tdE_{\pi(\mathcal{R})^{\frac{1}{\nu}}}(t)\widehat{f}^{(j-\ell)}(\pi)\Vert_{\textnormal{HS}}^2d\pi\right)^{\frac{1}{2}}\\
     &=A_{j}^{-1}\sup_{s\sim 2^{\ell}}\left(\smallint_{\widehat{G}}\Vert\widehat{K}(s\cdot \pi )\smallint_0^{\infty}\rho(2^{-j\nu}s^{\nu}t^{\nu})dE_{\pi(\mathcal{R})^{\frac{1}{\nu}}}(t)\widehat{f}^{(j-\ell)}(\pi)\Vert_{\textnormal{HS}}^2d\pi\right)^{\frac{1}{2}}\\
     &=A_{j}^{-1}\sup_{s\sim 2^{\ell}}\left(\smallint_{\widehat{G}}\Vert\widehat{K}(s\cdot \pi )\smallint_0^{\infty}\rho(2^{-j\nu}s^{\nu}t)dE_{\pi(\mathcal{R})}(t)\widehat{f}^{(j-\ell)}(\pi)\Vert_{\textnormal{HS}}^2d\pi\right)^{\frac{1}{2}}\\
     &=A_{j}^{-1}\sup_{s\sim 2^{\ell}}\left(\smallint_{\widehat{G}}\Vert\widehat{K}(s\cdot \pi )\rho(2^{-j}\cdot s\cdot \pi(\mathcal{R}))\widehat{f}^{(j-\ell)}(\pi)\Vert_{\textnormal{HS}}^2d\pi\right)^{\frac{1}{2}}\\
     &\leq A_{j}^{-1}\left(\smallint_{\widehat{G}}\sup_{\pi'\in \widehat{G}}\Vert\widehat{K}(\pi' )\rho(2^{-j}\cdot\pi'(\mathcal{R}))\widehat{f}^{(j-\ell)}(\pi)\Vert_{\textnormal{HS}}^2d\pi\right)^{\frac{1}{2}}\\
      &\leq A_{j}^{-1}\left(\smallint_{\widehat{G}}\sup_{\pi'\in \widehat{G}}\Vert\widehat{K}(\pi' )\rho(2^{-j}\cdot\pi'(\mathcal{R}))\Vert_{\textnormal{op}}^2\Vert\widehat{f}^{(j-\ell)}(\pi)\Vert_{\textnormal{HS}}^2d\pi\right)^{\frac{1}{2}}\\
      &\asymp \alpha_j A_{j}^{-1}\left(\smallint_{\widehat{G}}\Vert\widehat{f}^{(j-\ell)}(\pi)\Vert_{\textnormal{HS}}^2d\pi\right)^{\frac{1}{2}}\\
      &=\alpha_jA_j^{-1}\Vert \widehat{f}^{(j-\ell)}(\pi) \Vert_{L^2(\widehat{G})}.
\end{align*}All the analysis above allows us to write the estimates below
\begin{align*}
 &   \left\Vert \sup_{2^{\ell}\leq t<2^{\ell+1}}|f\ast (k_j)_t(x) |   \right\Vert_{L^2(G)}^2\\
 &\leq \sum_{\tau}\smallint_{\widehat{G}}  \Vert[\widehat{k}_{j}(t_{\tau}\cdot \pi)E_{j-\ell}(\pi)\widehat{f}(\pi)\Vert_{\textnormal{HS}}^2d\pi \\
 &+\left[2^\ell A_{j}^{-1}\sup_{s\sim 2^{\ell}}\left(\smallint_{\widehat{G}}\Vert\frac{d}{ds}[\widehat{k}_{j}(s\cdot \pi)E_{j-\ell}(\pi)\widehat{f}(\pi)]\Vert_{\textnormal{HS}}^2d\pi\right)^{\frac{1}{2}} \right]^2\\
 \\
 &\leq \sum_{\tau}\alpha_j^2 \Vert \widehat{f}^{(j-\ell)}(\pi)\Vert_{L^2(\widehat{G})}^2+\left[   A_{j}^{-1}\beta_j\Vert \widehat{f}^{(j-\ell)}(\pi)\Vert_{L^2(\widehat{G})}+ \alpha_jA_j^{-1}\Vert \widehat{f}^{(j-\ell)}(\pi) \Vert_{L^2(\widehat{G})}\right]^2\\
 &\lesssim A_j\alpha_j^2\Vert \widehat{f}^{(j-\ell)}(\pi)\Vert_{L^2(\widehat{G})}^2+A_jA_j^{-2}(\beta_j+\alpha_j)^2 \Vert \widehat{f}^{(j-\ell)}(\pi)\Vert_{L^2(\widehat{G})}^2\\
 &=(A_j\alpha_j^2+A_j^{-1}(\beta_j+\alpha_j)^2) \Vert \widehat{f}^{(j-\ell)}(\pi)\Vert_{L^2(\widehat{G})}^2.
\end{align*}
Therefore, we have proved that 
\begin{align*}
     \left\Vert \sup_{2^{\ell}\leq t<2^{\ell+1}}|f\ast (k_j)_t(x) |   \right\Vert_{L^2(G)}\lesssim\sqrt{  A_j\alpha_j^2+A_j^{-1}(\beta_j+\alpha_j)^2 }\Vert \widehat{f}^{(j-\ell)}(\pi)\Vert_{L^2(\widehat{G})}.
\end{align*}
Let us optimise the constant inside of the square root. It can be done by 
choosing  $$A_j:=\left[\alpha_j^{-1}(\beta_j+\alpha_j)\right],$$ to be the integer part of $\alpha_j^{-1}(\beta_j+\alpha_j)>1.$ It follows that $A_j\sim \alpha_j^{-1}(\beta_j+\alpha_j),$ and consequently   $$\sqrt{  A_j\alpha_j^2+A_j^{-1}(\beta_j+\alpha_j)^2 }\sim \alpha_j^{\frac{1}{2}}(\beta_j+\alpha_j)^{\frac{1}{2}}.$$
Now, we can return to the estimate of the maximal function. Application of the estimates above gives therefore 
\begin{align*}
    \Vert \sup_{t>0}|f\ast k_t|\Vert_{L^2(G)} &\leq \sum_{j\in \mathbb{Z}}\Vert \sup_{t>0}|f\ast (k_j)_t| \Vert_{L^2(G)}\leq \sum_{j\in \mathbb{Z}}\left\Vert \left(\sum_{\ell\in \mathbb{Z}} \sup_{2^{\ell}\leq t<2^{\ell+1}}|f\ast (k_j)_t|^2\right)^{\frac{1}{2}} \right\Vert_{L^2(G)}\\
    &=\sum_{j\in \mathbb{Z}}\left( \smallint_G  \sum_{\ell\in \mathbb{Z}} \sup_{2^{\ell}\leq t<2^{\ell+1}}|f\ast (k_j)_t(x)|^2 dx \right)^{\frac{1}{2}} \\
    &= \sum_{j\in \mathbb{Z}}\left(   \sum_{\ell\in \mathbb{Z}} \smallint_G \sup_{2^{\ell}\leq t<2^{\ell+1}}|f\ast (k_j)_t(x)|^2 dx \right)^{\frac{1}{2}}\\
    &\lesssim \sum_{j\in \mathbb{Z}}\left(   \sum_{\ell\in \mathbb{Z}} \alpha_j(\alpha_j+\beta_j) \Vert \widehat{f}^{(j-\ell)}(\pi)\Vert_{L^2(\widehat{G})}^2 \right)^{\frac{1}{2}}\\
    &= \sum_{j\in \mathbb{Z}} \alpha_j^{\frac{1}{2}}(\alpha_j+\beta_j)^{\frac{1}{2}} \left(   \sum_{\ell\in \mathbb{Z}}\Vert \widehat{f}^{(j-\ell)}(\pi)\Vert_{L^2(\widehat{G})}^2 \right)^{\frac{1}{2}}\\
    &= \sum_{j\in \mathbb{Z}} \alpha_j^{\frac{1}{2}}(\alpha_j+\beta_j)^{\frac{1}{2}} \left(   \sum_{m\in \mathbb{Z}}\Vert \widehat{f}^{(m)}(\pi)\Vert_{L^2(\widehat{G})}^2 \right)^{\frac{1}{2}}\\
    &\asymp \sum_{j\in \mathbb{Z}} \alpha_j^{\frac{1}{2}}(\alpha_j+\beta_j)^{\frac{1}{2}} \Vert f\Vert_{L^2(G)},
\end{align*}where in the last line we have use the almost orthogonality of the sequence $\widehat{f}^{(m)}(\pi),$ $m\in \mathbb{Z},$ on $L^2(\widehat{G})$ and the Plancherel theorem. The proof is complete.
\end{proof}
\begin{remark}\label{remark:2} We continue with the discussion of Remark \ref{remark}.
    Indeed, note that there exists $M=M(\phi,\eta)\in \mathbb{N},$ independent of $j\in \mathbb{Z},$ such that
    \begin{align*}
   \Vert \widehat{K}(\pi)\eta_j(\pi(\mathcal{R}))\Vert_{\textnormal{op}}\leq\sum_{k=j-M}^{j+M}   \Vert\widehat{K}(\pi)\phi_k(\pi(\mathcal{R}))\eta_j(\pi(\mathcal{R}))\Vert_{\textnormal{op}}\lesssim_{M}\sum_{k=j-M}^{j+M}\alpha_{k}^{(\phi)}\Vert \eta_j(\pi(\mathcal{R}))\Vert_{\textnormal{op}}. 
    \end{align*}
Since $\Vert \eta_j(\pi(\mathcal{R}))\Vert_{\textnormal{op}}\leq \Vert \eta\Vert_{L^\infty},$ and certainly $\alpha_{j}^{(\phi)}\asymp_{M} \sum_{k=j-M}^{j+M}\alpha_{k}^{(\phi)}, $ we have that $\alpha_j^{(\eta)}\lesssim \alpha_j^{(\phi)}.$ In the same way one can prove that  $\alpha_j^{(\phi)}\lesssim \alpha_j^{(\eta)}.$ In consequence $\alpha_j^{(\phi)}\asymp \alpha_j^{(\eta)}.$ The same argument proves that $\beta_j^{(\eta)}\asymp \beta_j^{(\phi)}.$
\end{remark}

\subsubsection*{Conflict of interests statement - Data Availability Statements}  The  author states that there is no conflict of interest.  Data sharing does not apply to this article as no datasets were generated or
analysed during the current study.

\subsection*{Acknowledgements}  The author is currently a FWO fellow supported by the Research Foundation Flanders FWO  under the postdoctoral
grant No 1204824N.

\bibliographystyle{amsplain}

\begin{thebibliography}{99}

\bibitem{Bourgain1985} Bourgain, J. Estimations de certaines fonctions maximales. C. R. Acad. Sci. Paris Sér. I Math. 301(10), 499--502, (1985). 

\bibitem{Bourgain1986} Bourgain, J. On high dimensional maximal functions associated to convex bodies, Amer. J. Math., 108, 1467--1476, (1986).

\bibitem{Bourgain2014} Bourgain, J. On the Hardy-Littlewood maximal function for the cube. Israel J. Math. 203, 275--293, (2014).

\bibitem{Bourgain1986:2}  Bourgain, J. Averages in the plane over convex curves and maximal operators, J. Anal. Math. 47, 69--85, (1986).

\bibitem{BourgainMirekSteinWróbel} Bourgain, J., Mirek, M., Stein, E. M., Wróbel, B. On dimension-free variational inequalities for averaging operators in Rd. Geom. Funct. Anal. 28(1), 58--99, (2018). 

\bibitem{Carbery1986} Carbery, A. An almost-orthogonality principle with applications to maximal functions associated to convex bodies. Bull. Amer. Math. Soc. 14(2), 269--274, (1986). 

\bibitem{Cardona2024:Jan} Cardona, D. Estimates for the full maximal function on graded Lie groups, submitted, arXiv:2401.07086.

\bibitem{CardonaDelgadoRuzhanskyDyadic} Cardona, D., Delgado, J., Ruzhansky, M.  Boundedness of the dyadic maximal function on graded Lie groups, to appear in Quart. J. Math.,  arXiv:2301.08964.

\bibitem{CorwinGreenleafBook}  Corwin, L. J., Greenleaf, F. P. Representations of nilpotent Lie groups and their applications. Part I. Basic theory and examples. Cambridge Studies in Advanced Mathematics, 18. Cambridge University Press, Cambridge, 1990. viii+269 pp. ISBN: 0-521-36034-X 

\bibitem{FischerRuzhanskyBook} Fischer V., Ruzhansky M. Quantization on nilpotent Lie groups, Progress in Mathematics, Vol. 314, Birkhauser, 2016. xiii+557pp.

\bibitem{FollandStein1982} Folland, G. Stein, E. Hardy Spaces on Homogeneous Groups, Princeton University
Press, Princeton, N.J., 1982.

\bibitem{Gromov:Milman} Gromov, M., Milman, V. Brunn theorem and a concentration of volume of convex bodies, GAFA Seminar Notes, Tel-Aviv, (1983-1984).

\bibitem{HormanderBook34} H\"ormander, L. The Analysis of the linear partial differential operators. Vol. III. IV, Springer-Verlag, (1985).

\bibitem{Kirillov1962} Kirillov, A.A. Unitary representations of nilpotent Lie groups. Russ. Math. Surv. 17, 57--110, (1962).

\bibitem{Stein1976} Stein, E. M. Maximal functions. I. Spherical means, Proc. Natl. Acad. Sci. USA 73, 2174--2175, (1976).

\bibitem{Stein1982} Stein, E. M. The development of square functions in the work of A. Zygmund, Bull. Amer. Math. Soc. 7, 359--376, (1982). 

\bibitem{Stein1983:1} Stein, E. M. The development of square functions in the work of A. Zygmund. Conference on harmonic analysis in honor of Antoni Zygmund, Vol. I, II (Chicago, Ill., 1981), 2--30, Wadsworth Math. Ser., Wadsworth, Belmont, CA, 1983.

\bibitem{Stein1983:2} Stein, E. M., Str\"omberg, J.-O. Behavior of maximal functions in Rn for large n. Ark. Mat. 21, no. 2, 259--269, (1983).

\end{thebibliography}

\end{document}